\documentclass[]{interact}

\usepackage{epstopdf}
\usepackage[caption=false]{subfig}

\usepackage[numbers,sort&compress]{natbib}
\bibpunct[, ]{[}{]}{,}{n}{,}{,}
\makeatletter
\def\NAT@def@citea{\def\@citea{\NAT@separator}}
\makeatother

\theoremstyle{plain}
\newtheorem{theorem}{Theorem}[section]

\newtheorem{corollary}[theorem]{Corollary}
\newtheorem{proposition}[theorem]{Proposition}

\theoremstyle{definition}
\newtheorem{definition}[theorem]{Definition}
\newtheorem{example}[theorem]{Example}

\theoremstyle{remark}
\newtheorem{remark}{Remark}

\usepackage [latin1]{inputenc}
\usepackage{mathabx}

\begin{document}


\title{Regularly abstract convex functions with respect to the set of Lipschitz continuous concave functions}

\author{
\name{Valentin~V. Gorokhovik\thanks{Email: gorokh@im.bas-net.by}}\thanks{ID: https://orcid.org/0000-0003-2447-5943}
\affil{Institute of Mathematics,
The National Academy of Sciences of Belarus, \\ Minsk, Belarus}
}

\maketitle

\begin{abstract}
Given a set ${\mathcal{H}}$ of functions defined on a set $X$, à function $f:X \mapsto {\overline{\mathbb{R}}}$ is called abstract ${\mathcal{H}}$-convex if it is the upper envelope of its ${\mathcal{H}}$-minorants, i.e., such its minorants which belong to the set ${\mathcal{H}}$; and $f$ is called regularly abstract ${\mathcal{H}}$-convex if it is the upper envelope of its maximal (with respect to the pointwise ordering) ${\mathcal{H}}$-minorants. In the paper we first present the basic notions of (regular) ${\mathcal{H}}$-convexity for the case when ${\mathcal{H}}$ is an abstract set of functions. For this abstract case a general sufficient condition based on Zorn's lemma for a ${\mathcal{H}}$-convex function to be regularly ${\mathcal{H}}$-convex is formulated.

The goal of the paper is to study the particular class of regularly ${\mathcal{H}}$-convex functions, when ${\mathcal{H}}$ is the set ${\mathcal{L}\widehat{C}}(X,{\mathbb{R}})$ of real-valued Lipschitz continuous classically concave functions defined on a real normed space $X$. For an extended-real-valued function $f:X \mapsto \overline{\mathbb{R}}$ to be ${\mathcal{L}\widehat{C}}$-convex it is necessary and sufficient that $f$ be lower semicontinuous and bounded from below by a Lipschitz continuous function; moreover, each ${\mathcal{L}\widehat{C}}$-convex function is regularly ${\mathcal{L}\widehat{C}}$-convex as well. We focus on ${\mathcal{L}\widehat{C}}$-subdifferentiability of functions at a given point. We prove that the set of points at which an ${\mathcal{L}\widehat{C}}$-convex function is ${\mathcal{L}\widehat{C}}$-subdifferentiable is dense in its effective domain. This result extends the well-known classical Br{\o}ndsted-Rockafellar theorem on the existence of the subdifferential for convex lower semicontinuous functions to the more wide class of lower semicontinuous functions. Using the subset ${\mathcal{L}\widehat{C}}_\theta$ of the set ${\mathcal{L}\widehat{C}}$ consisting of such Lipschitz continuous concave functions that vanish at the origin we introduce the notions of ${\mathcal{L}\widehat{C}}_\theta$-subgradient and ${\mathcal{L}\widehat{C}}_\theta$-subdifferential of a function at a point which generalize the corresponding notions of the classical convex analysis. Some properties and simple calculus rules for ${\mathcal{L}\widehat{C}}_\theta$-subdifferentials as well as ${\mathcal{L}\widehat{C}}_\theta$-subdifferential conditions for global extremum points are established.

Symmetric notions of abstract ${\mathcal{L}\widecheck{C}}$-concavity and ${\mathcal{L}\widecheck{C}}$-superdifferentiability of functions where ${\mathcal{L}\widecheck{C}}:= {\mathcal{L}\widecheck{C}}(X,{\mathbb{R}})$ is the set of Lipschitz continuous convex functions are also considered.

\end{abstract}


\begin{amscode}
52A01,49J52; 49K27; 26B40
\end{amscode}

\section{Introduction}

The paper relates to the intensively developed branch of modern nonsmooth analysis known as the theory of abstract convexity.
This theory originates in the works of Kutateladze and Rubinov \cite{KutRub72,KutRub}, published back in the early 1970s.
Some years later related ideas were discussed by Balder \cite{Balder} and by Dolecki and Kurcyusz \cite{Dolecki}. The results of a deep study of abstract convexity were summarized in the seminal monograph by Rubinov \cite{Rub2000} as well as in those by Pallaschke and Rolewicz \cite{PalRol} and by Singer \cite{Singer}.

A starting point of abstract convexity is the following result of classical convex analysis: a function $f$ defined on a real normed space is lower semicontinuous and convex if and only if it is the upper envelope of its continuous affine minorants.

The main idea of abstract convexity can be formulated as follows.
Given a set of functions $\mathcal{H}$ defined on an arbitrary set $X$, a function $f:X \mapsto \overline{\mathbb{R}}$ is called abstract convex with respect to the set $\mathcal{H}$ or, shortly, $\mathcal{H}$-convex if it can be presented as the upper envelope of a subset of functions from $\mathcal{H}$. Thus, in abstract convexity affine functions, used as elementary ones in classical convexity, are replaced by functions from a given set $\mathcal{H}$ but the operation of upper envelope is retained as a tool for producing abstract $\mathcal{H}$-convex functions. A family of elementary functions $\mathcal{H}$ can be given explicitly or by means of some another set $H$ and a coupling  function $\gamma$ on $X\times H$ pairing $X$ and $H$.
The class of abstract $\mathcal{H}$-concave functions is defined symmetrically by replacing the upper envelope with the lower envelope.
Different sets $\mathcal{H}$ of elementary functions generate different classes of abstract $\mathcal{H}$-convex functions.
The class of inf-convex functions coinciding with functions that are abstract concave with respect of the set of real-valued lower semicontinuous convex functions was described in \cite{GRC95}. In \cite[Chapters 2 and 3]{Rub2000} and \cite{Dutta1,Dutta2,Dutta3} the authors develop so called Monotonic Analysis which studies some classes of increasing functions, in particular, increasing and positively homogeneous (IPH) functions and increasing and convex-along-rays (ICAR) functions. It was shown that these classes of increasing functions are abstract convex with respect to suitable sets of elementary functions. In \cite{Martinez-Legaz_Singer} the authors have constructed a family $\mathcal{H}$ of extended-real-valued functions on ${\mathbb{R}}^n$ containing affine functions such that
an arbitrary extended-real-valued function (that is not necessarily lower semicontinuous) is convex if and only if it is $\mathcal{H}$-convex. The connection of abstract convex constructions with those of local non-smooth analysis was studied in \cite{Ioffe}.

Like the classical convex analysis the study of abstract convexity is mainly stimulated by numerous applications in optimization. In particular, in the framework of abstract convexity criteria for global minima and maxima  \cite{Rub2000,GorTyk2019-1,GorTyk2019-2,GorTyk2019-2a} as well as the duality results \cite{Balder,Flores-Bazan,Burachik_Rub07,Bedn_Syga2014,Bedn_Syga2018,BBKY} were derived.

The special class of abstract $\mathcal{H}$-convex functions called regularly abstract $\mathcal{H}$-convex ones was introduced in \cite{GorTyk2019-2,GorTyk2019-2a}.

A function $f:X \mapsto {\overline{\mathbb{R}}}$ is called regularly $\mathcal{H}$-convex if it can be presented as the upper envelope of its maximal (with respect to pointwise ordering) $\mathcal{H}$-minorants. Clearly, a regularly $\mathcal{H}$-convex function is $\mathcal{H}$-convex as well, but the converse as it is demonstrated by simple examples (see Example \ref{ex2.1} below) is not true in general. The fact that some abstract $\mathcal{H}$-convex functions can completely be characterized by their maximal $\mathcal{H}$-minorants was noted by Rubinov in \cite[Subsection 8.3]{Rub2000}.

In the present paper we continue to study regularly abstract $\mathcal{H}$-convex functions. In Section 2 we recall the basic notions of (regular) ${\mathcal{H}}$-convexity for the case when ${\mathcal{H}}$ is an abstract set of functions and formulate (Theorem 2.1) a general sufficient condition based on Zorn's lemma for a ${\mathcal{H}}$-convex function to be regularly ${\mathcal{H}}$-convex. The symmetric notions of abstract concave and regularly abstract concave functions with respect to some a set of elementary functions are discussed as well.

From Section 3 and later on we study the particular class of regularly ${\mathcal{H}}$-convex functions, when ${\mathcal{H}}$ is the set ${\mathcal{L}\widehat{C}}(X,{\mathbb{R}})$ of real-valued Lipschitz continuous concave (in classical sence) functions defined on a real normed space $X$. It was proved in \cite{GorTyk2019-2,GorTyk2019-2a}  that an extended-real-valued function $f:X \mapsto \overline{\mathbb{R}}$ is ${\mathcal{L}\widehat{C}}$-convex if and only if $f$ is lower semicontinuous and bounded from below by a Lipschitz continuous function; moreover, each ${\mathcal{L}\widehat{C}}$-convex function is regularly ${\mathcal{L}\widehat{C}}$-convex as well. It follows from this criterion that each lower semicontinuous convex function is ${\mathcal{L}\widehat{C}}$-convex as it is bounded from below by an affine continuous function. Thus the class of ${\mathcal{L}\widehat{C}}$-convex functions is a nontrivial extension of the class of lower semicontinuous convex functions. Further we focus on ${\mathcal{L}\widehat{C}}$-subdifferentiability of functions at a given point. We prove (Theorem \ref{th3.5}) that the set of points at which an ${\mathcal{L}\widehat{C}}$-convex function is ${\mathcal{L}\widehat{C}}$-subdifferentiable is dense in its effective domain. This result extends the well-known classical Br{\o}ndsted-Rockafellar theorem on the existence of the subdifferential for convex lower semicontinuous functions to the more wide class of lower semicontinuous functions.

In Section 4, using the subset ${\mathcal{L}\widehat{C}}_\theta$ of the set ${\mathcal{L}\widehat{C}}$ consisting of such Lipschitz continuous concave functions that vanish at the origin, we introduce the notions of (maximal) ${\mathcal{L}\widehat{C}}_\theta$-subgradient and (thin) ${\mathcal{L}\widehat{C}}_\theta$-subdifferential of a function at a point which generalize the corresponding notions of the classical convex analysis. Some properties of (thin) ${\mathcal{L}\widehat{C}}_\theta$-subdifferentials are presented in Proposition \ref{pr6.3}.

Symmetric notions of abstract ${\mathcal{L}\widecheck{C}}$-concavity and ${\mathcal{L}\widecheck{C}}$-superdifferentiability of functions where ${\mathcal{L}\widecheck{C}}:= {\mathcal{L}\widecheck{C}}(X,{\mathbb{R}})$ is the set of Lipschitz continuous convex functions are discussed in Section 5.

Simple calculus rules for ${\mathcal{L}\widehat{C}}_\theta$-subdifferentials as well as ${\mathcal{L}\widehat{C}}_\theta$-subdifferential conditions for global extremum points are established in Section 6.

In the paper we use mainly standard notations. In particular,
$\mathbb{R}$ is the set of real numbers, $\overline{\mathbb{R}}:={\mathbb{R}}\cup\{+\infty,-\infty\}$ is the set of extended real numbers, ${\mathbb{R}}_{+\infty} := {\mathbb{R}}\cup\{{+\infty}\}$, and ${\mathbb{R}}_{-\infty} :=  {\mathbb{R}}\cup\{{-\infty}\}$. The collection of all functions  $f:X \mapsto Z$ defined on a set $X$ and taking values in a set $Z$ is denoted by $Z^X$; below we consider mainly the cases when $Z$ coincides with one of the following sets: ${\mathbb{R}}$, ${\overline{\mathbb{R}}}, {\mathbb{R}}_{+\infty},$ or ${\mathbb{R}}_{-\infty}$.

\textit{The effective domain} of a function $f \in {\overline{\mathbb{R}}}^X$ is the set ${\rm dom}\,f := \{x \in X \mid |f(x)| < +\infty\}.$ A function $f \in {\overline{\mathbb{R}}}^X$ is called \textit{proper} if ${\rm dom}\,f \ne \varnothing$. Note that along with finite values a proper function can take both the value $+\infty$ and $-\infty$.

The epigraph of a proper function $f \in \overline{\mathbb{R}}^X$is the set
$${\rm epi}\,f :=\{(x,\gamma) \in X \times {\mathbb R} \mid f(x) \le \gamma\}$$
and the hypograph of $f$  is
$${\rm hypo}\,f :=\{(x,\gamma) \in X \times {\mathbb R} \mid f(x) \ge \gamma\}.$$

The collection ${\overline{\mathbb{R}}}^X$ is supposed to be ordered by the pointwise ordering $f \le g \Leftrightarrow f(x) \le g(x)\,\,\forall\,\,x \in X$. Observe, that $$f \le g\,\,\Leftrightarrow\,\,{\rm epi}\,g \subseteq {\rm epi}\,f\,\,\Leftrightarrow\,\,{\rm hypo}\,f \subseteq {\rm hypo}\,g.$$

\section{Preliminaries on ${\mathcal{H}}$-convex and regularly ${\mathcal{H}}$-convex functions}

In this section, for the convenience of readers, we recall basic definitions
 of abstract convexity and regularly abstract convexity, for more details we refer to the monographs \cite{Singer,PalRol,Rub2000} and to the papers \cite{GorTyk2019-2,GorTyk2019-2a}.

Throughout this section unless other specified $X$ is an abstract set, that is not equipped with any topological, algebraic or other structures; elements of $X$ will be called points.

Let us introduce and fix a set ${\mathcal{H}}:={\mathcal{H}}(X,{{\mathbb{R}_{- \infty}}})$ of proper extended-real-valued functions $h:X \mapsto {\mathbb{R}}_{-\infty} $ defined on $X$ and whose values are such that $h(x) < +\infty$ for all $x \in X$.

Given a proper function $f \in {\overline{\mathbb{R}}}^X$, the set $S^-({\mathcal{H}},f):=\{h \in \mathcal{H} \mid h \le f\}$ is called
\textit{the lower $\mathcal{H}$-support set} of $f$, while a functions $h$ from $S^-({\mathcal{H}},f)$ are called \textit{$\mathcal{H}$-minorants} of $f.$

A function $f:X \mapsto \overline{\mathbb{R}}$ is said to be \textit{abstract $\mathcal{H}$-convex} (or, shortly, \textit{$\mathcal{H}$-convex}), if $S^-({\mathcal{H}},f) \ne \varnothing$ and
\begin{equation}\label{e1.1}
f(x) = \sup\{h(x) \mid h \in S^-({\mathcal{H}},f)\}\,\,\text{for all}\,\,x \in X
\end{equation}
or, equivalently, if there is a nonempty subset $\mathcal{H'} \subset {\mathcal{H}}$ such that
\begin{equation}\label{e1.2}
f(x) = \sup\{h(x) \mid h \in \mathcal{H'}\}\,\,\text{for all}\,\,x \in X.
\end{equation}

It is easy to see that a subset $\mathcal{H'} \subset \mathcal{H}$ satisfying the equality \eqref{e1.2} belongs to $S^-({\mathcal{H}},f)$.

Evidently, that each function $h \in {\mathcal{H}}$ is abstract ${\mathcal{H}}$-convex because it satisfies \eqref{e1.2} with ${\mathcal{H'}} =\{h\}$. Through this we will refer to the set ${\mathcal{H}}$ as the set of elementary abstract convex functions. By analogy with the classical convex analysis the functions that belong ${\mathcal{H}}$ frequently call abstract affine functions.

By $S^{-}_{\text{max}}({\mathcal{H}},f)$ we denote the set of \textit{maximal $($with respect to the pointwise ordering$)$ $\mathcal{H}$-minorants of a function $f$}, i.e., the set of such functions $\bar{h} \in S^-({\mathcal{H}},f)$ that $$h \in S^-({\mathcal{H}},f),\,\bar{h} \le h \,\,\Longrightarrow h = \bar{h}.$$

Following \cite{GorTyk2019-2,GorTyk2019-2a} we call a function $f:X \mapsto \overline{\mathbb{R}}$ \textit{regularly  $\mathcal{H}$-convex} if $S^{-}_{\text{max}}({\mathcal{H}},f) \ne \varnothing$ and
\begin{equation}\label{e1.6}
f(x) = \sup\{h(x) \mid h \in {S^{-}_{\text{max}}({\mathcal{H}},f)}\}\,\,\text{for all}\,\,x \in X.
\end{equation}

The next example \cite{GorTyk2019-2,GorTyk2019-2a} shows that an ${\mathcal{H}}$-convex function may not be regularly ${\mathcal{H}}$-convex.
\begin{example}\cite{GorTyk2019-2,GorTyk2019-2a}\label{ex2.1}
Let $X = {\mathbb{R}}$  and let ${\mathcal{H}}$ be the set of linear functions with rational slopes, i.~e. ${\mathcal{H}} = \{x \mapsto qx \mid q \in {\mathbb{Q}}\}$, where ${\mathbb{Q}}$ is the set of rational numbers. Consider the function  $f:{\mathbb{R}} \mapsto {\mathbb{R}}$ such that  $f(x)=\sqrt{2}\,|x|$ for all $x \in {\mathbb{R}}.$ Then $S^-({\mathcal{H}},f) = \{x \mapsto qx \mid q \in {\mathbb{Q}},-\sqrt{2} < q < \sqrt{2}\}$. Since $f(x) = \sup\{qx \mid q \in {\mathbb{Q}},-\sqrt{2} < q < \sqrt{2}\}$ for all $x \in \mathbb{R}$, the function $f$ is ${\mathcal{H}}$-convex. At the same time, $S^-_{\text{max}}({\mathcal{H}},f) = \varnothing$ and, hence, the function $f$ is not regularly ${\mathcal{H}}$-convex.
\end{example}

\begin{remark}\label{rem2.1}
The fact that for some functions $f$ the lower support set $S^{-}({\mathcal{H}},f)$ can be completely described by the set $S^{-}_{\text{max}}({\mathcal{H}},f)\}$ of maximal elements was  noted by Rubinov in \cite[Subsection 8.3]{Rub2000}.
\end{remark}

From \eqref{e1.1} and \eqref{e1.6} we conclude that for an $\mathcal{H}$-convex function $f$ to be regularly $\mathcal{H}$-convex it is sufficient that $S^{-}_{\text{max}}({\mathcal{H}},f)$ satisfy the following property: for any ${\mathcal H}$-minorant $h \in S^-({\mathcal{H}},f)$ there exists a maximal ${\mathcal H}$-minorant $\bar{h} \in {S^{-}_{\text{max}}({\mathcal{H}},f)}$ such that $h \le \bar{h}.$
\vspace{3pt}

A subset ${\mathcal{C}} \subset {\mathcal{H}}$ is called \textit{a chain} if for any $h_1,\,h_2 \in {\mathcal{C}}$ either $h_1 \le h_2$ or $h_2 \le h_1$ holds.

\begin{theorem}\label{th2.1}
An $\mathcal{H}$-convex function $f:X \mapsto {\mathbb{R}}$ is regularly $\mathcal{H}$-convex if for any chain $\mathcal{C} \subset S^-({\mathcal{H}},f)$  there is $\bar{h} \in S^-({\mathcal{H}},f)$ such that $h \le \bar{h}$ for all $h \in {\mathcal{C}}.$ Moreover, in this case
for any $h \in S^-({\mathcal{H}},f)$ there exists $\bar{h} \in S^-_{\text{max}}({\mathcal{H}},f)$ such that $h \le \bar{h}.$
\end{theorem}

The proof follows from Zorn's lemma and the preceding remark.

\smallskip

Theorem \ref{th2.1} is conceptually related to Proposition 8.6 from \cite{Rub2000}.

\smallskip

An $\mathcal{H}$-minorant $h \in S^-({\mathcal{H}},f)$ of a function $f$ is said \textit{to support $f$ from below at a point} $\bar{x} \in {\rm dom}\,f$, if $h(\bar{x})=f(\bar{x}).$ The set of all $\mathcal{H}$-minorants of a function $f$ supporting $f$ from below at a point $\bar{x} \in {\rm dom}\,f$ is denoted by $S^-({\mathcal{H}},f,\bar{x})$.
For $\bar{x} \notin {\rm dom}\,f$ we define $S^-({\mathcal{H}},f,\bar{x}) = \varnothing$. Note, that for some $\bar{x} \in {\rm dom}\,f$ the set $S^-({\mathcal{H}},f,\bar{x})$ can be empty as well (see Example \ref{ex3.3}).

A function $f:X \mapsto {\mathbb{R}}$ is said \textit{to be ${\mathcal{H}}$-subdifferentiable at a point $\bar{x} \in {\rm dom}\,f$}, if  $S^-({\mathcal{H}},f,\bar{x}) \ne \varnothing$.

\smallskip

The set $S^-({\mathcal{H}},f,\bar{x})$
is nonempty if and only if
$$
f(\bar{x}) = \max\{h(\bar{x}) \mid h \in S^-({\mathcal{H}},f)\},
$$
with the maximum in the right-hand side of the latter equality being attained just at such functions $h$ from
$S^-({\mathcal{H}},f)$ that belong to $  S^-({\mathcal{H}},f,\bar{x})$.

The set of maximal $\mathcal{H}$-minorants of a function $f$ that support $f$ from below at a point $\bar{x} \in {\rm dom}\,f$ we denote by ${S^{-}_{\text{max}}({\mathcal{H}},f,\bar{x})}.$

\smallskip

The connections of the sets $S^{-}({\mathcal{H}},f,\bar{x})\}$ and $S^{-}_{\text{max}}({\mathcal{H}},f,\bar{x})\}$ with the abstract subdifferential of $f$ at $\bar{x}$ were discussed in \cite[Propositions 7.1 and 8.4]{Rub2000} (see, also, \cite{Bur_Rub08}).

\smallskip
In a similar way abstract concave functions are introduced. Here the basic definitions will be provided shortly.

Abstract concave functions are defined with respect to some set
${\mathcal{G}}:={\mathcal{G}}(X,{\mathbb{R}}_{+
\infty})$ of proper extended-real-valued functions defined on $X$ and whose values are such that $f(x) > -\infty$ for all $x \in X$.

The set $S^+({\mathcal{G}},f):=\{g \in \mathcal{G} \mid g \ge f\}$
is called \textit{upper} \textit{$\mathcal{G}$-support set} of a function $f \in {\overline{\mathbb{R}}}^X$, and functions $g$ from $S^+({\mathcal{G}},f)$ is called \textit{$\mathcal{G}$-majorants} of $f.$

A function $f: X \mapsto \overline{\mathbb{R}}$ is called (\textit{abstract}) \textit{$\mathcal{G}$-concave} if $S^+({\mathcal{G}},f) \ne \varnothing$ and
\begin{equation}\label{e1.7}
f(x) = \inf\{g(x) \mid g \in S^+({\mathcal{G}},f)\}\,\,\text{for all}\,\,x \in X.
\end{equation}

We refer to the set $\mathcal{G}$ as the set of elementary abstract concave functions.

\vspace{3pt}

By symbol $S^{+}_{\text{min}}({\mathcal{G}},f)$ we denote the set of minimal (with respect to the pointwise ordering) ${\mathcal{G}}$-majorants of the function $f$, i.e., the set of such functions $\bar{g} \in S^+({\mathcal{G}},f)$, that $g \in S^+({\mathcal{G}},f),\,\,\bar{g} \ge g\,\,\Longrightarrow g = \bar{g}.$

A function $f:X \mapsto \overline{\mathbb{R}}$ is called \textit{regularly $\mathcal{G}$-concave}, if $S^{+}_{\text{min}}({\mathcal{G}},f) \ne \varnothing$ and
$$
f(x) = \inf\{g(x) \mid g \in {S^{+}_{\text{min}}({\mathcal{G}},f)}\}\,\,\text{for all}\,\,x \in X.
$$

It follows from \eqref{e1.7} that for a $\mathcal{G}$-concave function $f$ to be regularly $\mathcal{G}$-concave, it is sufficient that  $S^{+}_{\text{min}}({\mathcal{G}},f) \ne \varnothing$  and for any function $g \in S^+({\mathcal{G}},f)$ there exist $\bar{g} \in {S^{+}_{\text{min}}({\mathcal{G}},f)}$ such that $g \ge \bar{g}.$

A $\mathcal{G}$-majorant $g \in S^+({\mathcal{G}},f)$ of a function $f$ is said \textit{to support} $f$ \textit{from above at a point} $\bar{x} \in {\rm dom}\,f$, if $g(\bar{x})=f(\bar{x}).$ The set of all $\mathcal{G}$-majorants of a function $f$, supporting $f$ from above at a point $\bar{x} \in {\rm dom}\,f$, is denoted by $S^+({\mathcal{G}},f,\bar{x})$.

A function $f:X \mapsto {\mathbb{R}}$ is said \textit{to be ${\mathcal{G}}$-superdifferentiable at a point $\bar{x} \in {\rm dom}\,f$}, if  $S^+({\mathcal{G}},f,\bar{x}) \ne \varnothing$.

The set $S^+({\mathcal{G}},f,\bar{x})$ is nonempty if and only if
$$
f(\bar{x}) = \min\{g(\bar{x}) \mid g \in S^+({\mathcal{G}},f)\}.
$$
The minimum in the right-hand side of the latter equality is attained just at such functions $g$ from $S^+({\mathcal{G}},f)$ that belong to $S^+({\mathcal{G}},f,\bar{x})$.

The set of minimal $\mathcal{G}$-majorants of a function  $f$ that support $f$ from above  at a point $\bar{x} \in {\rm dom}\,f$ will be denoted by ${S^{+}_{\text{min}}({\mathcal{G}},f,\bar{x})}.$

Using the sets $S^-_{\text{max}}({\mathcal{H}},f,\bar{x})$ of maximal ${\mathcal{H}}$-minorants and $S^{+}_{\text{min}}({\mathcal{G}},f,\bar{x})$ of minimal ${\mathcal{G}}$-majorants supporting $f$ at a point $\bar{x},$ one can derive a sufficient condition as well as a necessary one for points of global minimum and maximum of the function $f$ (see \cite[Theorems 1 and 2]{GorTyk2019-2a}).

\begin{remark}\label{r1}
It is easily seen from the definitions that when the sets ${\mathcal{G}}:={\mathcal{G}}(X,{\mathbb{R}}_{+
\infty})$ and ${\mathcal{H}}:={\mathcal{H}}(X,{\mathbb{R}}_{-
\infty})$ are symmetric, in the sense that $\mathcal{G} = -\mathcal{H}$, the notion of $\mathcal{G}$-concavity is symmetric to that of $\mathcal{H}$-convexity, i.e., a function $f$ is ${\mathcal{G}}$-concave if and only if $-f$ is ${\mathcal{H}}$-convex. When, in addition, ${\mathcal{H}} = - {\mathcal{H}}$, in particular, when ${\mathcal{H}}$ is a vector space in ${\mathbb{R}}^X$, a function $f$ is ${\mathcal{H}}$-concave if and only if $-f$ is ${\mathcal{H}}$-convex.
\end{remark}

The next example shows how the abstract concepts of ${\mathcal{H}}$-convexity and regular ${\mathcal{H}}$-convexity relate to the concepts of classical convexity.

\begin{example}\cite[Section 1]{GorTyk2019-2a}.\label{ex2.2}
Let $X$ be a real locally convex topological vector space, and let  ${\mathcal{H}}=~{\mathcal{A}}(X,\mathbb{R})$ be the vector space of continuous affine functions.

The following statements are equivalent
\begin{itemize}
\item[(i)] a function $f:X \mapsto {{\mathbb{R}}_{+\infty}}$ is lower (upper) semicontinuous on $X$ and convex (concave);
\item[(ii)] a function $f:X \mapsto {{\mathbb{R}}_{+\infty}}$ is ${\mathcal{A}}(X,\mathbb{R})$-convex (${\mathcal{A}}(X,\mathbb{R})$-concave); 
\item[(iii)] a function $f:X \mapsto {{\mathbb{R}}_{+\infty}}$ is regularly ${\mathcal{A}}(X,\mathbb{R})$-convex (regularly ${\mathcal{A}}(X,\mathbb{R})$-concave).
    \end{itemize}
For the proof of these equivalencies we refer to \cite[Section 1]{GorTyk2019-2a}.

\end{example}

\begin{example}\label{ex2.3}
Let $X$ be a real vector space, ${\widecheck{C}} := {\widecheck{C}}(X,\,{\mathbb{R}}_{+ \infty})$ the set of all convex functions $g:X \mapsto {\mathbb{R}}_{+ \infty}$ defined on $X$, and ${\widehat{C}} := {\widehat{C}}(X,\,{\mathbb{R}}_{- \infty})$ the set of all concave functions $h:X \mapsto {\mathbb{R}}_{- \infty}$ defined on $X$.

It follows from \cite[Theorem 3.2]{Gor2019} that any extended-real-valued function $f:X \to {\overline{\mathbb R}}$ is regularly $\widehat{C}$-convex and for each concave minorant $h \in S^-(\widehat{C},f)$ there exists a maximal concave minorant $\bar{h} \in S^-_{\text{max}}(\widehat{C},f)$ such that $h \le \bar{h}$. Moreover, $S^-_{\text{max}}(\widehat{C},f,\bar{x})$ is nonempty for all $\bar{x} \in {\rm dom}f$ and, consequently, any function $f:X \to {\overline{\mathbb R}}$ is $\widehat{C}$-subdifferentiable at every points $\bar{x} \in {\rm dom}f$.

Symmetrically, any extended-real-valued function $f:X \to {\overline{\mathbb R}}$ is regularly $\widecheck{C}$-concave and $\widecheck{C}$-superdifferentiable at every points $\bar{x} \in {\rm dom}f$.

\end{example}

\section{Regularly abstract convexity and subdifferentiability of functions with respect to Lipschitz continuous concave functions}

In this and next sections $X$ is a real normed vector space.

Recall (see, for instance, \cite{Martin}) that a function $h:X \mapsto {\overline{\mathbb{R}}}$ defined on a normed vector space $X$ is called \textit{globally Lipschitz continuous} or \textit{Lipschitz continuous on the whole space $X$} if there exists a real $k \ge 0$, called \textit{a Lipschitz constant}, such that
\begin{equation}\label{e3.1}
|h(x) - h(y)| \le k\|x - y\|\,\,\forall\,\,x,y \in X.
\end{equation}
Observe that the condition \eqref{e3.1} is equivalent to the inequality
\begin{equation}\label{e3.2}
    h(y) - k\|x - y\| \le h(x)\,\,\forall\,\,x,y \in X
\end{equation}
To see this one needs to swap places $x$ and $y$ in \eqref{e3.2}.

It follows immediately from the definition that any globally Lipschitz continuous function $h$ is either real-valued on $X$, i.e., such that ${\rm dom}\,h =X$, or $h \equiv \pm \infty.$

In what follows,  we exclude from consideration functions identically equal to infinity and consider only such globally Lipschitz functions which are real-valued.

If there is no risk of confusion, instead of saying `a function $h$ is globally Lipschitz continuous on X with Lipschitz constant $k \ge 0$' or `$h$ is globally Lipschitz  continuous on X', we will say simply `$h$ is $k$-Lipschitz continuous' or `$h$ is Lipschitz continuous'.

Let ${{\mathcal L}\widehat{C}}:={{\mathcal L}\widehat{C}}(X,{\mathbb{R}})$ be the set of real-valued Lipschitz continuous concave functions defined on $X$. In this section we focus on abstract convexity and subdifferentiability of functions with respect to the set ${{\mathcal L}\widehat{C}}:={{\mathcal L}\widehat{C}}(X,{\mathbb{R}})$.

 Let $f: X \mapsto \overline{\mathbb R}$ be an extended-real-valued function.

 Following the terminology of the general theory of abstract ${\mathcal H}$-convexity we call the set $S^-({{\mathcal L}\widehat{C}},f):=\{h \in {{\mathcal L}\widehat{C}} \mid h\le f\}$  \textit{the lower ${{\mathcal L}\widehat{C}}$-support of $f$}; and functions $h$ belonging to $S^-({{\mathcal L}\widehat{C}},f)$ are called \textit{${{\mathcal L}\widehat{C}}$-minorants of $f$}.

A function $f: X \mapsto \overline{\mathbb R}$ is called \textit{${{\mathcal L}\widehat{C}}$-convex} if $S^-({{\mathcal L}\widehat{C}},f) \ne \varnothing$ and $$f(x) =\sup\{h(x) \mid h \in S^-({{\mathcal L}\widehat{C}},f)\}.$$

Since $S^-({{\mathcal L}\widehat{C}},f)$ consists of real-valued functions, the condition $S^-({{\mathcal L}\widehat{C}},f) \ne \varnothing$ implies that $f(x) > -\infty$ for all $x \in X$.

\begin{theorem} {\rm \cite{GorTyk2019-2,GorTyk2019-2a}.}\label{th3.1}
 Any Lipschitz continuous function $f: X \mapsto {\mathbb{R}}$ is abstract ${{\mathcal L}\widehat{C}}$-convex.
\end{theorem}

\begin{proof}
Let $f: X \mapsto {\mathbb{R}}$ be a $k$-Lipschitz continuous function. Then it follows from the inequality \eqref{e3.2} that $f(x) = \max\{f(y) - k\|x-y\| \mid y \in X\}$. Since for every $y \in X$  the function $h_y(x):= f(y) - k\|x-y\|$ belongs to $S^-({{\mathcal L}\widehat{C}},f)$, the latter equality shows that the function $f$ is ${{\mathcal L}\widehat{C}}$-convex.
\end{proof}

\begin{remark}
 Another proof of the above assertion was given in \cite{GorTyk2019-2,GorTyk2019-2a}. It was based on the following criterion \cite[Theorem 4.1]{Gor2019}: a real-valued function $f:X \to {\mathbb R}$ is $k$-Lipschitz continuous if and only if each maximal concave minorant $($equivalently, each minimal convex majorant$)$ of $f$ also is Lipschitz continuous on $X$ with a Lipschitz constant not exceeding $k.$
\end{remark}

The next theorem shows that any Lipschitz continuous function is, actually, regularly abstract ${{\mathcal L}\widehat{C}}$-convex and, moreover, the class of regularly ${{\mathcal L}\widehat{C}}$-convex functions is essentially wider than the space of Lipschitz continuous functions.

\begin{theorem} {\rm \cite[Theorem 9]{GorTyk2019-2a}}. \label{th4.3}
For an arbitrary function $f: X \mapsto \overline{\mathbb{R}}$ the following three statements are equivalent:
\begin{itemize}
\item[$(i)$]
$f$ is ${\mathcal{L\widehat{C}}}$-convex;

\item[$(ii)$]
$f$ is regularly ${\mathcal{L\widehat{C}}}$-convex;

\item[$(iii)$]
$f$ is lower semicontinuous and bounded from below by a Lipschitz continuous function.
\end{itemize}
\end{theorem}

For the proof of this theorem we refer to \cite[Theorem 9]{GorTyk2019-2a}. Here we only note that (see \cite[Theorem 4.5]{Gor2019}) a function $f: X \mapsto \overline{\mathbb{R}}$ is bounded from below by a Lipschitz continuous function if and only if the set $S^-_{\text{max}}({\mathcal{L\widehat{C}}},f)$ of its maximal Lipschitz continuous concave minorants is nonempty. The presentation of $f$ in the form of the upper envelope of $S^-_{\text{max}}({\mathcal{L\widehat{C}}},f)$ is provided by the lower semicontinuity of $f$.

\vspace{1mm}

Observe that for any function $f: X \mapsto \overline{\mathbb{R}}$ its lower support set $S^-({\mathcal{L\widehat{C}}},f)$ is convex, while the set $S^-_{\text{max}}({\mathcal{L\widehat{C}}},f)$ of maximal ${\mathcal{L\widehat{C}}}$-minorants of $f$ is generally not convex.

\begin{example}\label{ex3.3}
Let us consider the function $f(x)=x^2.$ The functions $h_1(x) = 2x - 1$ and $h_2(x)=-2x-1$ are maximal ${\mathcal{L\widehat{C}}}$-minorants of $f$. Their convex combination $\frac{1}{2}h_1(x)+\frac{1}{2}h_2(x)\equiv -1$ is an ${\mathcal{L\widehat{C}}}$-minorant but not a maximal ${\mathcal{L\widehat{C}}}$-minorant of $f$.
\end{example}

\begin{theorem}\label{th4.4}
 A function $f:X \mapsto {\overline{\mathbb R}}$ is lower semicontinuous and convex in classical sense  if and only if it is regularly ${{\mathcal L}\widehat{C}}$-convex and each $h \in S^-_{{\rm max}}({{\mathcal L}\widehat{C}},f)$ is affine.
\end{theorem}

\begin{proof}
The ``only if'' part holds since any function that can be represented as the upper envelope of the family of continuous affine functions, is (see, for instance, \cite{ET76}) lower semicontinuous and convex in the classical sense.

For the ``if'' part suppose that $f:X \mapsto {\overline{\mathbb R}}$ is lower semicontinuous and convex in the classical sense. It follows from the separation theorems that $f$ is bounded from below by a continuous affine (and, consequently, Lipschitz continuous) function. Hence, due to Theorem \ref{th4.3} $f$ is regularly
${{\mathcal L}\widehat{C}}$-convex.

Further, for each maximal Lipschitz continuous concave minorant $h \in S^-_{\text{max}}({{\mathcal L}\widehat{C}},f)$ of the convex function $f$ we have $h \le f.$ By the sandwich theorem (see, for instant, \cite[Corollary 1.2.10]{BorVan10}) there exists a continuous affine function
$a: X \mapsto {\mathbb{R}}$ such that $h \le a \le f$. Through maximality of $h$ we have $h = a$.
\end{proof}

Thus, the class of ${{\mathcal L}\widehat{C}}$-convex functions is a nontrivial extension of the class of convex functions.

In addition we note that
if a function $f:X \mapsto {\overline{\mathbb R}}$ is lower semicontinuous and convex then
$
S^-_{\text{max}}({{\mathcal L}\widehat{C}},f)  = S^-_{\text{max}}({\mathcal{A}}(X,\mathbb{R}),f),
$
where ${\mathcal{A}}(X,\mathbb{R})$ stands for the space of continuous  affine functions defined on $X$.

Observe also that
$
S^-_{\text{max}}({\mathcal{A}}(X,\mathbb{R}),f)=\{x \mapsto x^*(x) - f^*(x^*) \mid x^* \in {\rm dom}\,f^*\}
$
where
$
f^*(x^*) := \sup\limits_{x \in X}\left(x^*(x) - f(x)\right), x^* \in X^*,
$
is the Fenchel conjugate \cite{ET76} of $f$.

Further we focus on ${\mathcal{L}\widehat{C}}$-subdifferentiability  of functions. Let us begin with recalling of the necessary definitions.

An ${{\mathcal L}\widehat{C}}$-minorant $h \in S^-({\mathcal L}\widehat{C},f)$ of the function $f:X \mapsto {\overline{\mathbb R}}$ is said \textit{to support $f$ from below at a point} $\bar{x} \in {\rm dom}\,f$, if $h(\bar{x})=f(\bar{x}).$

The set of all ${{\mathcal L}\widehat{C}}$-minorants, supporting $f$ from below at the point $\bar{x} \in {\rm dom}\,f$, is denoted by $S^-({{{\mathcal L}\widehat{C}}},f,\bar{x})$.

\begin{definition}\label{d5.1}
A function $f:X \mapsto \overline{\mathbb{R}}$ is called \textit{${\mathcal{L}\widehat{C}}$-subdifferentiable at a point $\bar{x} \in {\rm dom}\,f$} if $S^-({\mathcal{L}\widehat{C}},\,f,\,\bar{x}) \ne \varnothing$.
\end{definition}

By $S^-_{\text{max}}({{\mathcal L}\widehat{C}},f,\bar{x})$ we denote the subset of $S^-({{\mathcal L}\widehat{C}},f,\bar{x})$ consisting of all maximal ${{\mathcal L}\widehat{C}}$-minorants of $f$ supporting $f$ from below at the point $\bar{x}$. Clearly, that
$S^-_{\text{max}}({{\mathcal L}\widehat{C}},f,\bar{x}) = S^-_{\text{max}}({{\mathcal L}\widehat{C}},f) \bigcap S^-({{\mathcal L}\widehat{C}},f,\bar{x})$.

\vspace{3pt}

\begin{proposition}\label{pr5.1}
For a proper function $f:X \mapsto {\overline{\mathbb R}}$ and a point $\bar{x} \in {\rm dom}\,f$ one has
$$
S^-({{\mathcal L}\widehat{C}},f,\bar{x}) \ne \varnothing \Longleftrightarrow S^-_{\rm max}({{\mathcal L}\widehat{C}},f,\bar{x}) \ne \varnothing.
$$
Moreover,  for  every $h \in S^-({{\mathcal L}\widehat{C}},f,\bar{x})$ there exists $\bar{h} \in S^-_{\rm max}({{\mathcal L}\widehat{C}},f,\bar{x})$ such that $h \le \bar{h}.$
\end{proposition}

\begin{proof}
The implication $S^-_{\text{max}}({{\mathcal L}\widehat{C}},f,\bar{x}) \ne \varnothing \Longrightarrow S^-({{\mathcal L}\widehat{C}},f,\bar{x}) \ne \varnothing$ is obvious.

Suppose that $S^-({{\mathcal L}\widehat{C}},f,\bar{x}) \ne \varnothing$ and consider an arbitrary $h \in S^-({{\mathcal L}\widehat{C}},f,\bar{x}).$ Since $h$ is a concave minorant of $f$ then (see Example \ref{ex2.3}) there exists a maximal concave minorant $\bar{h}$ of $f$ such that $h \le \bar{h} \le f.$ It follows from \cite[Lemma 4.4.]{Gor2019} that the function $\bar{h}$ is Lipschitz continuous and, consequently, from the equality $h(\bar{x}) = f(\bar{x})$ we get that $\bar{h} \in S^-_{\text{max}}({{\mathcal L}\widehat{C}},f,\bar{x})$. Thus, $S^-_{\text{max}}({{\mathcal L}\widehat{C}},f,\bar{x}) \ne \varnothing$ and the converse implication also holds.

Since $h$ is an arbitrary element of $S^-({{\mathcal L}\widehat{C}},f,\bar{x})$ the last claim of the proposition is true as well.
\end{proof}

\begin{corollary}\label{cor5.1a}
A function $f$ is ${\mathcal{L}\widehat{C}}$-subdifferentiable at a point $\bar{x} \in {\rm dom}\,f$ if and only if $S^-_{{\rm max}}({\mathcal{L}\widehat{C}},\,f,\,\bar{x}) \ne \varnothing$.
\end{corollary}

In other words, a function $f$ is ${\mathcal{L}\widehat{C}}$-subdifferentiable at a point $\bar{x} \in {\rm dom}\,f$ if and only if there exists a (maximal) ${\mathcal{L}\widehat{C}}$-minorant of $f$ that supports $f$ from below at $\bar{x}.$

\begin{theorem}\label{th5.1}
An extended-real-valued function $f:X \mapsto \overline{\mathbb{R}}$ is \textit{${\mathcal{L}\widehat{C}}$-subdifferentiable at a point $\bar{x} \in {\rm dom}\,f$} if and only if there exists a constant $k \ge 0$ such that
\begin{equation}\label{e3.4a}
f(x) \ge f(\bar{x}) - k\|x-\bar{x}\|\,\,\forall\,\,x \in X.
\end{equation}
\end{theorem}

\begin{proof}
Let $f:X \mapsto \overline{\mathbb{R}}$ be ${\mathcal{L}\widehat{C}}$-subdifferentiable at a point $\bar{x} \in {\rm dom}\,f$. It means that $S^-({{{\mathcal L}\widehat{C}}},f,\bar{x})$ is nonempty. Choose an arbitrary $h \in S^-({{{\mathcal L}\widehat{C}}},f,\bar{x}).$ Since $h$ is Lipschitz continuous, there exists  $k \ge 0$ such that $|h(x)-h(y)|\le k\|x- y\|$ for all $x,y \in X.$ Substituting $\bar{x}$ for $y$ into the latter inequality and taking into account that $h \le f$ and $h(\bar{x})=f(\bar{x})$, we arrive at \eqref{e3.4a}.

Suppose now, that for some $k \ge 0$ the inequality \eqref{e3.4a} holds. Since the function $\tilde{h}: x \mapsto f(\bar{x}) -k\|x-\bar{x}\|$ is a Lipschitz continuous concave minorant of $f$ and $\tilde{h}(\bar{x})=f(\bar{x}),$ we conclude that $\tilde{h} \in S^-({{{\mathcal L}\widehat{C}}},f,\bar{x})$ and, consequently, $S^-({{\mathcal L}\widehat{C}},f,\bar{x}) \ne \varnothing$. Hence, the function $f:X \mapsto \overline{\mathbb{R}}$ is ${\mathcal{L}\widehat{C}}$-subdifferentiable at a point $\bar{x} \in {\rm dom}\,f$.
\end{proof}

\begin{remark}\label{r5.3}
Theorem \ref{th5.1} shows that ${{\mathcal L}\widehat{C}}$-subdifferentiability of a function $f$ at a point $\bar{x}$ is closely related to the property called the calmness of $f$ at $\bar{x}$
(see, for instance, \cite{RW1998,Penot2013}).
\end{remark}

\begin{corollary}\label{cor5.1}
If a function $f:X \mapsto {\mathbb{R}}$ is Lipschitz continuous on $X$ then $f$ is ${\mathcal{L}\widehat{C}}$-subdifferentiable at every point $\bar{x} \in X$.
\end{corollary}

\proof The claim is immediate  from Theorem \ref{th5.1}. 
\hfill $\Box$

\begin{proposition}\label{pr5.2}
Let $f:X \mapsto \overline{\mathbb{R}}$ be ${\mathcal{L}\widehat{C}}$-subdifferentiable at a point $\bar{x} \in {\rm dom}\,f$.
Then
\begin{equation}\label{e5.2a}
{\mathcal A}(X,{\mathbb{R}}) \bigcap S^-({{\mathcal L}\widehat{C}},f,\bar{x}) \subset S^-_{\rm max}({{\mathcal L}\widehat{C}},f,\bar{x}).
\end{equation}
Here ${\mathcal A}(X,{\mathbb{R}})$ stands for the vector space of continuous affine functions defined on $X$.

In other words, any continuous affine minorant supporting $f$ from below at some point $\bar{x} \in {\rm dom}\,f$ is a maximal element in $S^-({{\mathcal L}\widehat{C}},f,\bar{x})$.
\end{proposition}

\proof If ${\mathcal A}(X,{\mathbb{R}}) \bigcap S^-({{\mathcal L}\widehat{C}},f,\bar{x}) = \varnothing$, the inclusion \eqref{e5.2a} is trivial. To prove \eqref{e5.2a} for the case when ${\mathcal A}(X,{\mathbb{R}}) \bigcap S^-({{\mathcal L}\widehat{C}},f,\bar{x}) \ne \varnothing$ we assume the contrary and take some $a \in {\mathcal A}(X,{\mathbb{R}}) \bigcap S^-({{\mathcal L}\widehat{C}},f,\bar{x})$ that does not belong to $S^-_{\text{max}}({{\mathcal L}\widehat{C}},f,\bar{x}).$ Then there exists $\bar{h} \in S^-_{\text{max}}({{\mathcal L}\widehat{C}},f,\bar{x})$ such that $a \le \bar{h}.$ Since $\bar{h}$ is concave and Lipschitz continuous, there is $\bar{a} \in {\mathcal A}(X,{\mathbb{R}})$ that satisfies $\bar{h} \le \bar{a}$ and $\bar{h}(\bar{x}) = \bar{a}(\bar{x}).$ Thus, we have $a \le \bar{h} \le \bar{a}$ and, in addition, $a(\bar{x}) = \bar{h}(\bar{x}) = \bar{a}(\bar{x}) = f(\bar{x}).$  Consequently, $a =\bar{h} = \bar{a}$, that contradicts the assumption $a \notin S^-_{\text{max}}({{\mathcal L}\widehat{C}},f,\bar{x})$. \hfill $\Box$

\begin{theorem}\label{th3.5}
Let $X$ be a complete normed vector space.Then for any ${{\mathcal L}\widehat{C}}$-convex function $f:X \mapsto \overline{\mathbb{R}}$ the set of points at which $f$ is ${\mathcal{L}\widehat{C}}$-subdifferentiable  is dense in ${\rm dom}\,f.$
\end{theorem}

\proof Since the function $f$ is ${\mathcal{L}\widehat{C}}$-convex, the set  $S^-({\mathcal{L}\widehat{C}},f)$ consisting of ${\mathcal{L}\widehat{C}}$-minorants of $f$ is nonempty and
\begin{equation}\label{e3.5a}
f(x)= \sup\{h(x) \mid h \in S^-({\mathcal{L}\widehat{C}},f)\}\,\,\text{for all}\,\,x \in X.
\end{equation}

Consider an arbitrary point $\bar{x} \in {\rm dom}\,f.$
It follows from the equality \eqref{e3.5a} that for any $\varepsilon > 0$ there exists a function $h \in S^-({\mathcal{L}\widehat{C}},f)$ such that $h(\bar{x})+\varepsilon \ge f(\bar{x})$. Since $h$ is Lipschitz continuous and by Theorem \ref{th4.3} $f$ is lower semicontinuous the function $g: x \mapsto f(x)-h(x)$ also is lower semicontinuous. Moreover, $g(x) \ge 0\,\,\forall\,\,x \in X$, and, consequently, $g$ is bounded from below. In addition, we have $g(\bar{x}) \le \varepsilon \le \inf\limits_{x \in X}g(x) + \varepsilon.$ Now, it follows from Ekeland's variational principle \cite{ET76,Penot2013,Kruger,Ioffe2017,Mord1} that for an arbitrary $\delta > 0$ there exists a point $x_\delta \in X$ such that
\begin{itemize}
\item[$(i)$]
$g(x_\delta) + \displaystyle\frac{\varepsilon}{\delta}\|x_\delta-\bar{x}\| \le g(\bar{x});$
\item[$(ii)$]
$\|x_\delta-\bar{x}\| \le \delta;$
\item[$(iii)$]
$g(x_\delta) < g(x) + \displaystyle\frac{\varepsilon}{\delta}\|x-x_\delta\|\,\,\text{for all}\,\,x \in X,x \ne x_\delta.$
\end{itemize}

The property $(iii)$ can be rewritten in the form
$$
f(x) > h(x) - \displaystyle\frac{\varepsilon}{\delta}\|x-x_\delta\|+(f(x_\delta) -h(x_\delta))\,\,\forall\,\,x \in X,\,x \ne x_\delta.
$$

The function $\bar{h}:x \mapsto h(x) - \displaystyle\frac{\varepsilon}{\delta}\|x-x_\delta\|+(f(x_\delta)-h(x_\delta))$ is concave and Lipschitz continuous on $X$ 
and, in additition, $\bar{h}(x_\delta)=f(x_\delta).$ Hence, the function $\bar{h}$ is ${\mathcal{L}\widehat{C}}$-minorant of $f$ which supports $f$ at the point $x_\delta$ and, consequently, the function $f$ is ${\mathcal{L}\widehat{C}}$-subdifferentiable at the point $x_\delta$.

Furthermore, through the conditions $(i)$ and $(ii)$ the point $x_\delta$ belongs to $({\rm dom}\,f)\cap B_\delta(\bar{x}),$ where $B_\delta(\bar{x}):=\{x \in X \mid \|x_\delta - \bar{x}\| \le \delta\}.$
Since the choice of the point $\bar{x} \in {\rm dom}\,f$ and a real $\delta > 0$ was arbitrary, we conclude that any neighborhood of each point from ${\rm dom}\,f$ contains such points at which the function $f$ is ${\mathcal{L}\widehat{C}}$-subdifferentiable. \hfill $\Box$

\begin{remark}\label{r3.2}
{\rm Theorem \ref{th3.5} extends the well-known Br{\o}ndsted--Rockafellar theorem \cite{BrRock65} on
the existence of the subdifferential for convex (in the classical sense) lower semicontinuous functions to the wider class
of lower semicontinuous functions.}
\end{remark}

In contrast to classically convex functions, even in the finite-dimensional setting, an ${\mathcal{L}\widehat{C}}$-convex function may not be ${\mathcal{L}\widehat{C}}$-subdifferentiable at some interior points of its effective domain.

\begin{example}\label{ex3.3}
The function $f(x) = - \sqrt{|x|}$ is continuous on ${\mathbb{R}}$ and bounded from below by the Lipschitz continuous function $h(x)= -|x|-1$, and, hence, by Theorem \ref{th4.3} $f$ is ${\mathcal{L}\widehat{C}}$-convex. It is easily seen that $f$ is not ${\mathcal{L}\widehat{C}}$-subdifferentiable at the point $x=0$.
\end{example}

\section{${\mathcal{L}\widehat{C}}_\theta$-subgradients and ${\mathcal{L}\widehat{C}}_\theta$-subdifferential of a function.}

\vspace{1mm}

The aim of this section is to introduce notions of the ${\mathcal{L}\widehat{C}}_\theta$-subgradient and the ${\mathcal{L}\widehat{C}}_\theta$-subdifferential of a function $f$ at a point $\bar{x}$. For this we need some additional preliminaries on the space ${\mathcal{L}}(X,\mathbb{R})$ of Lipschitz continuous functions defined on a normed space~$X$.

Recall that all Lipschitz continuous functions we consider are real-valued.

A real-valued function $f:X \mapsto {\mathbb{R}}$ is Lipschitz continuous if and only if
$$
\|f\|_{\mathcal{L}}:= \sup_{x,y \in X,\,x \ne y}\frac{|f(x)-f(y)|}{\|x - y\|} < +\infty.
$$
The value $\|f\|_{\mathcal{L}}$, called \textit{the Lipschitz modulus of} $f$, is the infimum of the set of constants $k \ge 0$ such that
$$
|f({x}) - f(y)| \le k\|x - y\|\,\,\forall\,\,x,y \in X.
$$
holds.

A function $f \mapsto \|f\|_{\mathcal{L}}$ is a semi-norm on the vector space ${\mathcal{L}}(X,\mathbb{R})$, with $\|f  - h\|_{\mathcal{L}} = 0$ if and  only if $f(x) -h(x) \equiv\,\,\text{const}\,\,\forall\,\,x \in X.$

Fix $\bar{x} \in X$ and define for $f \in {\mathcal{L}}(X,\mathbb{R})$ the value
$$
\|f\|_{\mathcal{L},\,\bar{x}}:= |f(\bar{x})| + \|f\|_{\mathcal{L}}.
$$
Then the function $\|\cdot\|_{\mathcal{L},\,\bar{x}}: x \mapsto |f(\bar{x})| + \|f\|_{\mathcal{L}}$ is a norm on ${\mathcal{L}}(X,\mathbb{R})$ with respect to which the normed vector space $({\mathcal{L}}(X,\mathbb{R}),\|\cdot\|_{\mathcal{L},\,\bar{x}})$ is complete.

By ${\mathcal{L}}_{\bar{x}}(X,\mathbb{R})$ we denote the vector subspace in ${\mathcal{L}}(X,\mathbb{R})$ consisting of such functions $f \in {\mathcal{L}}(X,\mathbb{R})$ that $f(\bar{x})= 0.$ Since for each function $f \in {\mathcal{L}}_{\bar{x}}(X,\mathbb{R})$ we have $\|f\|_{\mathcal{L},\,\bar{x}} = \|f\|_{\mathcal{L}}$, the function $\|\cdot\|_{\mathcal{L}}$ is a norm on ${\mathcal{L}}_{\bar{x}}(X,\mathbb{R})$ with respect to which ${\mathcal{L}}_{\bar{x}}(X,\mathbb{R})$ is a complete normed space.

Suppose that a sequence of Lipschitz continuous functions
$\{f_n(\cdot)\}_{n=1}^\infty$ from the Banach space ${\mathcal{L}}_{\bar{x}}(X,\mathbb{R})$ converges to a function $f(\cdot)$ in  the norm $\|\cdot\|    _{\mathcal L}$. Then for any $x \in X$  the sequence of reals  $\{f_n(x)\}_{n=1}^\infty$ converges to $f(x)$ (i.e., the sequence $\{f_n(\cdot)\}_{n=1}^\infty$ converges to $f(\cdot)$ pointwise on $X$). Moreover, in this case the sequence $\{f_n(\cdot)\}_{n=1}^\infty$ converges to $f(\cdot)$ uniformly on each bounded subset of $X$.

Indeed, since the norm space $({\mathcal{L}}_{\bar{x}}(X,\mathbb{R}),\|\cdot\|_{\mathcal{L}})$ is complete, then $f \in {\mathcal{L}}_{\bar{x}}(X,\mathbb{R}).$ Therefore, we get the inequality
$$
|f_n(x)-f(x)| = |(f_n-f)(x) - (f_n-f)(\bar{x})| \le \|f_n-f\|_{{\mathcal{L}}} \|x - \bar{x}\|,
$$
which shows that the sequence of functions $\{f_n(\cdot)\}_{n=1}^\infty$ converges to the function $f(\cdot)$ both pointwise and uniformly on each bounded subset of $X$.

\vspace{3pt}

For proofs of above statements and more details on the space of Lipschitz functions we refer to \cite{Martin,Hein,PalRol}.

\vspace{3pt}

In what follows we will assume that $\bar{x} = \theta$, where $\theta$ is the origin of $X$, and will deal with the Banach spaces $({\mathcal{L}}(X,\mathbb{R}),\,\|f\|_{\mathcal{L},\,\theta})$ and $({\mathcal{L}}_\theta(X,\mathbb{R}),\,\|f\|_{\mathcal{L}}).$

Note, that we can uniquely associate each function $f \in {\mathcal{L}}(X,\mathbb{R})$ and an arbitrary point $(\bar{x},f(\bar{x})) \in X \times {\mathbb R}$ of its graph  with the function $l: x \mapsto f(x + \bar{x}) - f(\bar{x})$ from ${\mathcal{L}}_\theta(X,\mathbb{R})$ such that $f(x) = l(x - \bar{x}) + f(\bar{x})$ for all $x \in X.$ Conversely, each function $l \in {\mathcal{L}}_\theta(X,\mathbb{R})$ and a point $(\bar{x},c) \in X \times {\mathbb R}$
are uniquely associated with the function $f: x \mapsto l(x - \bar{x}) + c$ from ${\mathcal{L}}(X,\mathbb{R})$ such that $f(\bar{x}) = c.$

Below, along with the convex cone ${\mathcal L}\widehat{C}:= {\mathcal L}\widehat{C}(X,{\mathbb{R}})$ of Lipschitz continuous concave functions we will consider its convex subcone ${{\mathcal L}\widehat{C}}_\theta:={{\mathcal L}\widehat{C}}_\theta(X,{\mathbb{R}})$ consisting of such functions $l \in {\mathcal L}\widehat{C}(X,{\mathbb{R}})$ for which the equality  $l(\theta) = 0$ holds. Clearly, ${{\mathcal L}\widehat{C}}_\theta(X,{\mathbb{R}})$ is a subset of the Banach space $({\mathcal{L}}_\theta(X,\mathbb{R}),\,\|\cdot\|_{\mathcal{L}}).$

\begin{definition}\label{d6.1}
A function $l \in {{\mathcal L}\widehat{C}}_\theta(X,{\mathbb{R}})$ is called \textit{an ${{\mathcal L}\widehat{C}}_\theta$-subgradient of a function $f:X \mapsto \overline{\mathbb{R}}$ at a point $\bar{x} \in {\rm dom}\,f$} if
\begin{equation}\label{e6.7}
f(x) - f(\bar{x}) \ge l(x - \bar{x})\,\,\text{for all}\,\,x \in X.
\end{equation}
The set of all ${{\mathcal L}\widehat{C}}_\theta$-subgradients of a function $f:X \mapsto \overline{\mathbb{R}}$ at a point $\bar{x} \in {\rm dom}\,f$  denoted by $D_{{{\mathcal L}\widehat{C}}_\theta}f(\bar{x})$ is called \textit{the ${{{\mathcal L}\widehat{C}}_\theta}$-subdifferential of $f$ at $\bar{x}$}.

\vspace{3pt}

An ${{\mathcal L}\widehat{C}}_\theta$-subgradient $l \in D_{{{\mathcal L}\widehat{C}}_\theta}f(\bar{x})$ that is a maximal (with respect to the pointwise ordering) element of $D_{{{\mathcal L}\widehat{C}}_\theta}f(\bar{x})$ is called \textit{a maximal ${{\mathcal L}\widehat{C}}_\theta$-subgradient of $f$ at ${\bar{x}}$}.
The set of all maximal ${{\mathcal L}\widehat{C}}_\theta$-subgradients of a function $f:X \mapsto \overline{\mathbb{R}}$ at a point $\bar{x} \in {\rm dom}\,f$,  denoted by $\partial_{{{\mathcal L}\widehat{C}}_\theta}f(\bar{x})$, is called \textit{thin ${{{\mathcal L}\widehat{C}}_\theta}$-subdifferential of $f$ at $\bar{x}.$}
\end{definition}

\begin{proposition}\label{pr6.2}
A function $l \in {{\mathcal L}\widehat{C}}_\theta(X,{\mathbb{R}})$ is an ${{\mathcal L}\widehat{C}}_\theta$-subgradient of a function $f:X \mapsto \overline{\mathbb{R}}$ at a point $\bar{x} \in {\rm dom}\,f$ if and only if the function $h: x \mapsto l(x-\bar{x}) + f(\bar{x})$ belongs to $S^-({\mathcal{L}\widehat{C}},\,f,\,\bar{x})$, while  a function $l \in {{\mathcal L}\widehat{C}}_\theta(X,{\mathbb{R}})$ is  \textit{a maximal ${{\mathcal L}\widehat{C}}_\theta$-subgradient of a function $f:X \mapsto \overline{\mathbb{R}}$ at a point $\bar{x} \in {\rm dom}\,f$} if and only if  the function $h: x \mapsto l(x-\bar{x}) + f(\bar{x})$ belongs to $S^-_{\text{max}}({\mathcal{L}\widehat{C}},\,f,\,\bar{x})$.

\end{proposition}

\begin{proof}
The first assertion of the proposition follows immediately from the inequality \eqref{e6.7}.

Observe that due to the definitions of $\partial_{{{\mathcal L}\widehat{C}}_\theta}f(\bar{x})$ and $D_{{{\mathcal L}\widehat{C}}_\theta}f(\bar{x})$ there are the one-to-one correspondences between $\partial_{{{\mathcal L}\widehat{C}}_\theta}f(\bar{x})$ and $S^-_{\text{max}}({\mathcal{L}\widehat{C}},\,f,\,\bar{x})$ as well as between $D_{{{\mathcal L}\widehat{C}}_\theta}f(\bar{x})$ and $S^-({\mathcal{L}\widehat{C}},\,f,\,\bar{x})$ preserving order relations between their elements. The second assertion follows from this observation.
\end{proof}

\begin{example}\label{ex6.2}
{\rm For the Lipschitz continuous function $f:{\mathbb R}^2 \to {\mathbb R}$ such that $f(x,y)=|x|-|y|$, its thin ${{\mathcal L}{\widehat{C}}}_{\theta}$-subdifferential at the point $(0,0)$ is equal to
$$\partial_{{\mathcal L}{\widehat{C}}_{\theta}}f(0,0) = \{h:{\mathbb{R}}^2 \mapsto {\mathbb{R}} \mid h_\alpha:(x,y) \mapsto \alpha x -|y|,\, -1\le \alpha \le 1\}.$$ In particular,
the functions $h_{-1}(x,y) := -x - |y|$, $h_{0}(x,y) :=  - |y|$ and $h_{1}(x,y) := x - |y|$
are maximal ${{\mathcal L}{\widehat{C}}}_{\theta}$-subgradients of the function $f$ at the point $(0,0)$.

The ${{\mathcal L}{\widehat{C}}_{\theta}}$-subdifferential
 $D_{{\mathcal L}{\widehat{C}}_{\theta}}f(0,0)$ consists of all Lipschitz concave functions $h:{\mathbb{R}}^2 \mapsto {\mathbb{R}}$ such that $h(0,0) =0$ and $h(x,y) \le |x| -|y|$ for all $(x,y) \in {\mathbb{R}}^2.$
}
\end{example}

\begin{proposition}\label{pr6.3}
Let $f \in {\overline{\mathbb{R}}^X}$ and let $\bar{x} \in {\rm dom}\,f$.
\begin{itemize}
\item[$(i)$]
The ${{{\mathcal L}\widehat{C}}_\theta}$-subdifferential of $f$ at  $\bar{x}$, $D_{{{\mathcal L}\widehat{C}}_\theta}f(\bar{x})$, is a convex and closed subset of the Banach space $({\mathcal{L}}_\theta(X,\mathbb{R}),\,\|\cdot\|_{\mathcal{L}}).$
\item[$(ii)$]
Let a normed vector space $X$ be complete and let a function $f:X \mapsto \overline{\mathbb{R}}$ be ${{\mathcal L}\widehat{C}}$-convex. Then the set of those $x \in X$ 
at which $\partial_{{{\mathcal L}\widehat{C}}_\theta}f({x})$ $($and, consequently, $D_{{{\mathcal L}\widehat{C}}_\theta}f({x})$$)$ is nonempty, is dense in ${\rm dom}\,f$.
\item[$(iii)$]
 Let a function $f:X \mapsto \overline{\mathbb{R}}$ be lower semicontinuous and convex in classical sense. If  $\partial_{{{\mathcal L}\widehat{C}}_\theta}f(\bar{x}) \ne \varnothing$ at a point $\bar{x} \in {\rm dom}\,f$ then every maximal ${{\mathcal L}\widehat{C}}_\theta$-subgradient of the function $f$  at the point $\bar{x}$ is, actually, a continuous linear function, and, consequently,
$$
\partial_{{{\mathcal L}\widehat{C}}_\theta}f(\bar{x}) = \{x^* \in X^* \mid f(x)-f(\bar{x}) \ge x^*(x-\bar{x})\,\,\forall\,\,x \in X\} =: \partial f(\bar{x}),
$$
where $\partial f(\bar{x})$ is the Fenchel--Rockafellar subdifferential of $f$ at a point $\bar{x} \in X.$
\end{itemize}
\end{proposition}

\begin{proof}
The claim $(i)$ is easily verified from the inequality \eqref{e6.7}.

The claim $(ii)$ is actually another formulation of Theorem \ref{th3.5}.

At last, the claim $(iii)$ follows from Theorem \ref{th4.4}.
\end{proof}

\begin{proposition}\label{pr6.4}
If for a function $f \in {\overline{\mathbb{R}}^X}$ there exists a point $\bar{x} \in {\rm dom}\,f$ such that
$$
X^*\bigcap\partial_{{{\mathcal L}{\widehat{C}}_\theta}}f(\bar{x}) \ne \varnothing\,\,\text{and}\,\,X^*\bigcap\partial_{{{\mathcal L}{\widehat{C}}_\theta}}(-f)(\bar{x}) \ne \varnothing,
$$
then $f$ is a continuous affine function and, in this case,
$$
\partial_{{{\mathcal L}{\widehat{C}}_\theta}}f(\bar{x}) = -\partial_{{{\mathcal L}{\widehat{C}}_\theta}}(-f)(\bar{x}) = \{x^*\}
$$
where $x^* \in X^*$ is a continuous linear functional such that $f(x)\!=\!x^*(x\! -\! \bar{x})\!+\!f(\bar{x}),\,x \in X.$
\end{proposition}

\begin{proof}
Let $x^*_1 \in X^*\bigcap\partial_{{{\mathcal L}{\widehat{C}}_\theta}}f(\bar{x})$ and $-x^*_2 \in X^*\bigcap\partial_{{{\mathcal L}{\widehat{C}}_\theta}}(-f)(\bar{x})$. Then the following inequalities hold
\begin{equation}\label{e6.8a}
x^*_1(x - \bar{x})+ f(\bar{x}) \le f(x) \le x^*_2(x - \bar{x})+ f(\bar{x})\,\,\forall\,\,x \in X,
\end{equation}
from which we conclude that
$x^*_1(x) \le x^*_2(x)\,\,\forall\,\,x \in X.$
It implies $x_1^* = x_2^* =: x^*$. Due to this fact we get from \eqref{e6.8a} that  $f(x) =x^*(x-\bar{x}) + f(\bar{x})\,\,\forall\,\,x \in X$ and, hence, $\partial_{{{\mathcal L}{\widehat{C}}_\theta}}f(\bar{x}) = -\partial_{{{\mathcal L}{\widehat{C}}_\theta}}(-f)(\bar{x}) = \{x^*\}$. \end{proof}

\section{ Abstract ${\mathcal{L}\widecheck{C}}$-concave functions and ${\mathcal{L}\widecheck{C}}_\theta$-superdifferentiability.}

\vspace{1mm}

Using the set ${\mathcal{L}\widecheck{C}}:={\mathcal{L}\widecheck{C}}(X,{\mathbb{R}})$ of Lipschitz continuous and convex in classical sense functions as the set of elementary abstract concave functions and the operation of pointwise infimum as a tool for producing new functions from subsets of ${\mathcal{L}\widecheck{C}}$ we can define the class of ${\mathcal{L}\widecheck{C}}$-concave functions. We briefly formulate the main definitions and some results related to the abstract ${\mathcal{L}}\widecheck{C}$-concavity of functions.

A function $f:X \mapsto \overline{\mathbb{R}}$ is called \textit{abstract ${\mathcal{L}}\widecheck{C}$-concave} if the set $S^+({\mathcal{L}}\widecheck{C},\,f) :=\{g \in {\mathcal{L}}\widecheck{C} \mid g \ge f\}$ consisting of ${\mathcal{L}}\widecheck{C}$-majorant of the function $f$ is non-empty and
$$
f(x) = \inf\{g(x) \mid g \in S^+({\mathcal{L}}\widecheck{C},\,f)\}\,\, \text{for all}\,\, x \in X,
$$
or, equivalently, if in ${\mathcal{L}}\widecheck{C}$  there exists such non-empty subset $\mathcal{F}$ that
$$
f(x) = \inf\{g(x) \mid g \in {\mathcal{F}}\}\,\, \text{for all}\,\, x \in X.
$$

Since ${\mathcal{L}}\widecheck{C} = -{\mathcal{L}}\widehat{C}$, it is easily seen that a function $f$ is ${\mathcal{L}}\widecheck{C}$-concave if and only if $-f$ is ${\mathcal{L}}\widehat{C}$-convex. Thus, like classical notions of convexity and concavity of functions abstract ${\mathcal{L}}\widehat{C}$-convexity an abstract ${\mathcal{L}}\widecheck{C}$-concavity are symmetric notions too.

Due to this symmetry we conclude from Theorem \ref{th4.3} that a function $f$ is ${\mathcal{L}}\widecheck{C}$-concave if and only if it is upper semicontinuous and bounded from above by a Lipschitz continuous function.

Hence, the class of functions which are both ${\mathcal{L}}\widehat{C}$-convex and ${\mathcal{L}}\widecheck{C}$-concave coincides with the class of continuous functions bounded from below and from above by Lipschitz continuous functions. In particular, Lipschitz continuous functions are both ${\mathcal{L}}\widehat{C}$-convex and ${\mathcal{L}}\widecheck{C}$-concave.

We say that a function $f:X \mapsto \overline{\mathbb{R}}$ is \textit{${\mathcal{L}}\widecheck{C}$-superdifferentiable at a point $\bar{x} \in {\rm dom}\,f$} if there exists such its ${\mathcal{L}}\widecheck{C}$-majorant $g \in S^+({\mathcal{L}}\widecheck{C},f)$ that $g(\bar{x}) = f(\bar{x}).$

It is easily seen that a function $f$ is ${\mathcal{L}}\widecheck{C}$-superdifferentiable at a point $\bar{x} \in {\rm dom}\,f$ if $-f$ is ${\mathcal{L}}\widehat{C}$-subdifferentiable at  $\bar{x}$, so the notion of ${\mathcal{L}}\widecheck{C}$-superdifferentiability is also symmetric to the notion of ${\mathcal{L}}\widehat{C}$-subdifferentiability.

A function $l \in {\mathcal{L}}\widecheck{C}_\theta := \{l \in {\mathcal{L}}\widecheck{C} \mid l(\theta) = 0\}$ is called \textit{an ${{\mathcal{L}}\widecheck{C}}_\theta$-supergradient of a function $f:X \mapsto \overline{\mathbb{R}}$ at a point $\bar{x} \in {\rm dom}\,f$} if the function $g(x):=l(x - \bar{x}) + f(\bar{x})\,\,\forall\,\,x \in X$ belongs to $S^+({\mathcal{L}}\widecheck{C},f),$  i.e., if
$$
f(x) - f(\bar{x}) \le l(x-\bar{x})\,\,\text{for all}\,\,x \in X.
$$

The set of all ${{\mathcal{L}}\widecheck{C}}_\theta$-supergradients of a function $f$ at a point $\bar{x} \in {\rm dom}\,f$ is called \textit{the ${{{\mathcal{L}}\widecheck{C}}_\theta}$-superdifferential of $f$ at $\bar{x}$} and is denoted by $D^+_{{{\mathcal{L}}\widecheck{C}}_\theta}f(\bar{x}),$ while the subset of minimal (with respect to pointwise ordering) elements of $D^+_{{{\mathcal{L}}\widecheck{C}}_\theta}f(\bar{x})$  will be called \textit{the thin ${{\mathcal{L}}\widecheck{C}}_\theta$-superdifferential of $f$ at $\bar{x}$} and will be denoted by $\partial^+_{{{\mathcal{L}}\widecheck{C}}_\theta}f(\bar{x}).$

The equalities
$$
D^+_{{{\mathcal{L}}\widecheck{C}}_\theta}f(\bar{x}) = - D_{{{\mathcal{L}}\widehat{C}}_\theta}(-f)(\bar{x})\,\,\text{and}\,\,\partial^+_{{{\mathcal{L}}\widecheck{C}}_\theta}f(\bar{x}) = - \partial_{{{\mathcal{L}}\widehat{C}}_\theta}(-f)(\bar{x}).
$$
show that there is the symmetry between (thin) ${{{\mathcal{L}}\widehat{C}}_\theta}$-subdifferentials and (thin) ${{{\mathcal{L}}\widecheck{C}}_\theta}$-superdifferentials.

\begin{example}\label{ex7.1}
{\rm Consider the Lipschitz continuous function $f:{\mathbb R}^2 \to {\mathbb R}$ such that $f(x,y)=|x|-|y|$ (see Example \ref{ex6.2}). Its thin ${{\mathcal L}{\widecheck{C}}}_{\theta}$-superdifferential at the point $(0,0)$ is equal to
$$\partial^+_{{\mathcal L}{\widecheck{C}}_{\theta}}f(0,0) = \{g:{\mathbb{R}}^2 \mapsto {\mathbb{R}} \mid g_\alpha:(x,y) \mapsto |x| - \alpha y,\, -1\le \alpha \le 1\}.$$

\smallskip

Note that for this function both the thin ${{\mathcal L}{\widehat{C}}}_{\theta}$-subdifferential and the thin ${{\mathcal L}{\widecheck{C}}}_{\theta}$-superdifferential at the point $(0,0)$ are non-empty. At the same time, both
the Fr\'{e}chet subdifferential and the Fr\'{e}chet superdifferential of $f$ at $(0,0)$ are empty. In general, as it  is known  \cite{Kruger}, for any function $f$ and a point $\bar{x} \in {\rm dom}\,f$ such that $f$ is not Fr\'{e}chet differentiable at $\bar{x}$ only one of the sets, either the Fr\'{e}chet subdifferential or the Fr\'{e}chet superdifferential of $f$ at $\bar{x}$, can be non-empty.
}
\end{example}

\section{Simple calculus rules for ${\mathcal{L}\widehat{C}}_\theta$-subdifferentials and ${\mathcal{L}\widehat{C}}_\theta$-subdifferential conditions for points of global extremum.}

\vspace{1mm}

The following simple rules calculus for ${\mathcal{L}\widehat{C}}_\theta$-subdifferentials follow immediately from their definitions and properties

\begin{proposition}\label{pr8.1}
Let $f: X \mapsto \overline{\mathbb{R}}$ and $-f: X \mapsto \overline{\mathbb{R}}$ be ${{\mathcal L}{\widehat{C}}}$-subdifferentiable at a point $\bar{x} \in {\rm dom}\,f.$  Then
$$D_{{\mathcal L}{\widehat{C}}_\theta}(\lambda f)(\bar{x}) = \begin{cases} \lambda D_{{\mathcal L}{\widehat{C}}_\theta}f(\bar{x}), & \text{when $\lambda > 0,$}\\ \lambda D^+_{{\mathcal L}{\widecheck{C}}_\theta}f(\bar{x}), & \text{when $\lambda < 0$};\end{cases}$$
and
$$\partial_{{\mathcal L}{\widehat{C}}_\theta}(\lambda f)(\bar{x}) = \begin{cases} \lambda \partial_{{\mathcal L}{\widehat{C}}_\theta}f(\bar{x}), & \text{when $\lambda > 0,$}\\ \lambda \partial^+_{{\mathcal L}{\widecheck{C}}_\theta}f(\bar{x}), & \text{when $\lambda < 0$};\end{cases}$$
\end{proposition}

\begin{proposition}\label{pr8.2}
Let $f_1,f_2:X \mapsto {\overline{\mathbb{R}}}$ be ${{\mathcal L}{\widehat{C}}}$-subdifferentiable at a point $\bar{x} \in X$. Then $f_1+f_2$ is ${{\mathcal L}{\widehat{C}}}$-subdifferentiable at the point $\bar{x}$, and
$$D_{{\mathcal L}{\widehat{C}}_\theta}f_1(\bar{x}) + D_{{\mathcal L}{\widehat{C}}_\theta}f_2(\bar{x}) \subseteq D_{{\mathcal L}{\widehat{C}}_\theta}(f_1 + f_2)(\bar{x})
$$
$(ii)$ for any  $l_1 \in \partial_{{\mathcal L}{\widehat{C}}_\theta}f_1(\bar{x})$ and any $l_2 \in \partial_{{\mathcal L}{\widehat{C}}_\theta}f_2(\bar{x})$ there exists $l \in \partial_{{\mathcal L}{\widehat{C}}_\theta}(f_1 + f_2)(\bar{x})$ such that $l(x) \ge l_1(x) + l_2(x)$ for all $x \in X.$
\end{proposition}

\begin{proposition}\label{pr8.3}
Let $f_1,f_2:X \mapsto {\overline{\mathbb{R}}}$ be ${{\mathcal L}{\widehat{C}}}$-subdifferentiable at a point $\bar{x} \in {\rm dom}\,f_1\cap{\rm dom}\,f_2$. If $f_1(x) \le f_2(x)$ for all $x \in X$ and $f_1(\bar{x}) = f_2(\bar{x})$, then for each $l_1 \in D_{{{\mathcal L}{\widehat{C}}}_\theta}f_1(\bar{x})$ there exists $l_2 \in \partial_{{{\mathcal L}{\widehat{C}}}_\theta}f_2(\bar{x})$ such that $l_1(x) \le l_2(x)$ for all $x \in X.$
\end{proposition}

Using symmetry between ${{\mathcal L}{\widehat{C}}_\theta}$-subdifferentials and ${{\mathcal L}{\widecheck{C}}_\theta}$-superdifferentials, similar calculus rules can be formulated for ${{\mathcal L}{\widecheck{C}}_\theta}$-superdifferentials.

\begin{theorem}[\rm{{global extremum criterium}}]\hspace{-6pt}\hspace{3pt}

$(i)$ If a function $f:X \mapsto \overline{\mathbb{R}}$ is ${{{\mathcal L}{\widehat{C}}}}$-subdifferentiable at a point $\bar{x} \in {\rm dom}\,f$,
then  $f$ achieves its global minimum over $X$ at the point $\bar{x} \in {\rm dom}\,f$ if and only if
$$
 0_{X^*} \in \partial_{{{\mathcal L}{\widehat{C}}_\theta}}f(\bar{x}).
$$
Here $\partial_{{{\mathcal L}{\widehat{C}}_\theta}}f(\bar{x})$ stands for the thin ${{{\mathcal L}{\widehat{C}}_\theta}}$-subdifferential of $f$ at $\bar{x}$, and $0_{X^*}$ is the null linear function defined on $X$ $($the null element of $X^*$$)$.

$(ii)$ If a function $f:X \mapsto \overline{\mathbb{R}}$ is ${{{\mathcal L}{\widecheck{C}}_\theta}}$-superdifferentiable at a point $\bar{x} \in {\rm dom}\,f$,
then  $f$ achieves its global maximum over $X$ at the point $\bar{x} \in {\rm dom}\,f$ if and only if
$$
 0_{X^*} \in \partial^+_{{{\mathcal L}{\widecheck{C}}_\theta}}f(\bar{x}).
$$
Here $\partial^+_{{{\mathcal L}{\widecheck{C}}_\theta}}f(\bar{x})$ is the thin ${{{\mathcal L}{\widecheck{C}}_\theta}}$-superdifferential of $f$ at $\bar{x}$.
\end{theorem}

\begin{proof}
The statement $(i)$ immediates from the definitions of global minimum and from the inequality \eqref{e6.7} which characterizes ${{{\mathcal L}{\widehat{C}}_\theta}}$-subgradients of $f$ at $\bar{x}$ as well as from that the null linear function belongs to ${{\mathcal L}\widehat{C}}_\theta(X,{\mathbb{R}})$.

The proof of $(ii)$ is similar.
\end{proof}

\begin{theorem}[\rm{necessary condition of global extremum}]\hspace{-6pt}\hspace{3pt}

$(i)$ Assume that a function $f:X \mapsto \overline{\mathbb{R}}$ is ${\mathcal{L}}\widehat{C}$-subdifferentiable at a point $\bar{x} \in {\rm dom}\,f$.  If the function $f$ attains its global maximum over $X$ at the point $\bar{x}$, then
$$
0_{X^*} \in \partial^+ h(\bar{x})\,\,\text{for all}\,\,h \in S^{-}({\mathcal{L}}\widehat{C},f,\bar{x}).
$$
Here $S^{-}({\mathcal{L}}\widehat{C},f,\bar{x})$ is the set of ${\mathcal{L}}\widehat{C}$- minorants, supporting $f$ from below at $\bar{x}$, $\partial^+ h(\bar{x}):= \{x^* \in X^* \mid x^*(x-\bar{x}) \ge h(x) - h(\bar{x})\,\,\forall\,\,x \in X\}$ is the classical Moreau-Rockafellar superdifferential of a concave function $h$ at a point $\bar{x},$ and $0_X^*$ is the null linear functional defined on $X$.

$(ii)$ Let a function $f:X \mapsto \overline{\mathbb{R}}$ be ${\mathcal{L}}\widecheck{C}$-superdifferentiable at a point $\bar{x} \in {\rm dom}\,f$. If the function $f$ attains its global maximum over $X$ at the point $\bar{x}$, then
$$
0_{X^*} \in \partial g(\bar{x})\,\,\text{for all}\,\,g \in S^{+}({\mathcal{L}}\widecheck{C},f,\bar{x}).
$$
Here $S^{+}({\mathcal{L}}\widecheck{C},f,\bar{x})$ is the set of ${\mathcal{L}}\widehat{C}$- majorants, supporting $f$ from above at $\bar{x}$, $\partial g(\bar{x}):= \{x^* \in X^* \mid x^*(x-\bar{x}) \le g(x) - g(\bar{x})\,\,\forall\,\,x \in X\}$ is the classical Moreau-Rockafellar subdifferential of a convex function $g$ at a point $\bar{x}.$

\end{theorem}

\begin{proof}
$(i)$ If the function $f$ attains its global maximum over $X$ at the point $\bar{x}$, then each function $h \in S^{-}({\mathcal{L}}\widehat{C},f,\bar{x})$ attains its global maximum at the point $\bar{x}$ as well. Since functions $h \in S^{-}({\mathcal{L}}\widehat{C},f,\bar{x})$ is concave in conventional sense, this is equivalent the condition $0_{X^*} \in \partial^+ h(\bar{x})$.

By similar arguments we come to $(ii)$.
\end{proof}

\section{Conclusion}

In conclusion, briefly about some directions for future research.

(1) It is of interest to study the behavior of ${\mathcal{L}}\widehat{C}$-subdifferential $D^-_{{{\mathcal L}{\widehat{C}}_\theta}}f(\bar{x})$ and that of the thin ${{{\mathcal L}{\widehat{C}}_\theta}}$-subdifferential  $\partial^+_{{{\mathcal L}{\widecheck{C}}_\theta}}f(\bar{x})$ of a function $f$ at $\bar{x} \in {\rm dom}f$ under varying $\bar{x}$.

\smallskip

(2) The concept of ${{{\mathcal L}{\widehat{C}}}}$-subdifferentiability of a function $f$ at a point $x$, introduced in this article, is global, since it requires the existence of a global minorant of the function $f$ which supports from below $f$ at a point $x$. It would be reasonable to localize this notion, that would extend the class of ${{{\mathcal L}{\widehat{C}}}}$-subdifferentiable functions.

\smallskip

(3) For a real $\varepsilon > 0$ the concepts of $\varepsilon$-${\mathcal{L}}\widehat{C}$-subdifferential and  the thin $\varepsilon$-${\mathcal{L}}\widehat{C}$-subdifferential can also be introduced and explored.

\section*{Acknowledgement}

The author thanks the referees for the careful reading of the manuscript and their constructive comments and suggestions.

\section*{Disclosure statement}

There are no conflicts of interest to disclose.

\section*{Funding}

The research was supported by the State Program for Fundumental Research of Republic of Belarus 'Convergence-2025'.


\begin{thebibliography}{99}

\bibitem{KutRub72}
Kutateladze~SS., Rubinov~AM. Minkowski duality and its applications. Russian Math Surveys. 1972;27(3):137--191. 

\bibitem{KutRub}
Kutateladze~SS, Rubinov~AM. Minkowski Duality and Its Applications.  Novosibirsk: Nauka; 1976. (in Russian)

\bibitem{Balder}
Balder EJ. An Extension of Duality-Stability Relations to Nonconvex Optimization Problems. SIAM Journal on Control and Optimization. 1977;15(2):329--343. doi:10.11370315022

\bibitem{Dolecki}
Dolecki~S, Kurcyush~S. On $\Phi$-convexity in extremal problems. SIAM J. Control and Optimization.  1978;16(2):277--300.

\bibitem{Rub2000}
Rubinov~AM. Abstract convexity and global optimization. Dordrecht: Kluwer Academic Publishers; 2000. 

\bibitem{PalRol}
Pallaschke~D, Rolewicz~S. Foundations of Mathematical Optimization (Convex analysis without linearity). Dordrecht: Kluwer Academic Publishers; 1997. 
doi: 10.1007/978-94-017-1588-1\,.

\bibitem{Singer}
Singer~I. Abstract Convex Analysis. New York (NY): Wiley-Interscience
Publication; 1997. 

\bibitem{GRC95}
Glover BM, Rubinov AM, Craven BD. Solvability theorems involving inf-convex functions. Journal of Mathematical Analysis and Applications. 1995;191(2):305--330.

\bibitem{Dutta1}
Dutta J, Martinez-Legaz JE, Rubinov AM. Monotonic analysis over cones: I. Optimization. 2004;53(2):129--146.

\bibitem{Dutta2}
Dutta J, Martinez-Legaz JE, Rubinov AM. Monotonic analysis over cones: II. Optimization. 2004;53(5-6):529--547.

\bibitem{Dutta3}
Dutta J, Martinez-Legaz JE, Rubinov AM. Monotonic analysis over cones: III. J. Convex Analysis. 2008;15(3):561--579.

\bibitem{Martinez-Legaz_Singer}
Martinez-Legaz JE, Singer I. On $\phi$-convexity of convex functions. Linear Algebra and Its Applications. 1998;278(1-3):163--181.

\bibitem{Ioffe}
Ioffe AD. Abstract convexity and non-smooth analysis. Advances in  Math.  Econ. 2001;3:45-61.

\bibitem{GorTyk2019-1}
Gorokhovik VV, Tykoun AS. Support points of semicontinuous functions with respect to Lipschitz concave functions. Doklady Nats. Acad. of Belarus. 2019;63(6):647--653. (in Russian) https://doi.org/10.29235/1561-8323-2019-63-6-647-653

\bibitem{GorTyk2019-2}
Gorokhovik VV, Tykoun AS. Abstract convexity of functions with respect to
the set of Lipschitz (concave) functions. Trudy Instituta Matematiki i Mekhaniki URO RAN. 2019;25(3):73--85. (in Russian)

\bibitem{GorTyk2019-2a}
Gorokhovik VV, Tykoun AS. Abstract convexity of functions with respect to
the set of Lipschitz (concave) functions. Proceedings of the Steklov Institute of Mathematics. 2020;309(l):S36--S46.

\bibitem{Flores-Bazan}
Flores-Baz\'{a}n F. On a notion of subdifferentiability for non-convex functions. Optimization. 1995;33(1):1--8. doi10.108002331939508844059

\bibitem{Burachik_Rub07}
Burachik RS, Rubinov AM. Abstract convexity and augmented Lagrangians. SIAM Journal on Optimization. 2007;18(2):413--436. https://doi.org/10.1137050647621

\bibitem{Bedn_Syga2014}
Bednarczuk EM, Syga M. Minimax theorems for $\phi$-convex functions
with applications. Control and Cybernetics. 2014;43(3):421--437.

\bibitem{Bedn_Syga2018}
Bednarczuk EM, Syga M. On minimax theorems for lower semicontinuous functions in Hilbert spaces. Journal of Convex Analysis. 2018;25(2):389--402.

\bibitem{BBKY}
Bui TH, Burachik RS, Kruger AY, Yost DT.  Zero duality gap conditions via abstract convexity. Optimization. 2021. 37p. https://doi.org/10.1080/02331934.2021.1910694.

\bibitem{Bur_Rub08}
Burachik-Rubinov On abstract convexity and set valued analysis. Journal of Nonlinear and Convex Analysis. 2008;9(1):105--123.

\bibitem{ET76}
Ekeland~I, Temam~R. Convex Analysis and Variational Problems.
Amsterdam: North-Holland, 1976.

%

\bibitem{Gor2019}
Gorokhovik VV. Minimal convex majorants of functions and Demyanov--Rubinov exhaustive super(sub)differentials. Optimization. J. Math. Programming Operations Research. 2018;68(10):1933--1961.
doi 10.1080/02331934.2018.1518446

\bibitem{Martin}
Martin~RH, Nonlinear operators and differential equations in Banach spaces. New York: Wiley; 1976.

\bibitem{BorVan10}
Borwein JM, Vanderwerff JD. Convex functions: constructions, characterizations and   counterexamples. Cambridge: Cambridge University Press; 2010.

\bibitem{RW1998}
Rockafellar RT, Wets RJ-B. Variational analysis. Berlin: Springer; 1998.

\bibitem{Kruger}
Kruger~AY, On Fr\'{e}chet subdifferentials. Journal of Mathematical Sciences. 2003;116(3):3325--3358.

\bibitem {Ioffe2017}
Ioffe~AD. Variational Analysis of Regular Mappings. Theory and Applications. Springer International Publishing AG; 2017.

\bibitem{Mord1}
Mordukhovich~BS. Variational Analysis and Generalized Differentiation. I: Basic Theory. Berlin: Springer; 2005.

\bibitem{Penot2013}
Penot~JP. Calculus without Derivatives. New York (NY): Springer; 2013.

\bibitem{BrRock65}
Br{\o}ndsted~A, Rockafellar~RT. On the subdifferentiability
of convex functions. Proc. Amer. Math. Soc. 1965;16(4):605--611.

\bibitem{Hein}
Heinonen~J. Lectures on analysis on metric spaces. New York: Springer; 2001. 

\end{thebibliography}
\end{document}